\documentclass[12pt, twoside]{article} %{amsart}
\usepackage{amssymb,latexsym}
\usepackage{enumerate}
\usepackage{verbatim}
\usepackage{amsfonts}
\usepackage{amsmath}
\usepackage{fancyhdr}
\usepackage{amsmath,amsxtra,amssymb,amsthm,latexsym,amscd}
\usepackage{array}
\usepackage[dvips]{graphicx}

\numberwithin{equation}{section}

\theoremstyle{plain}
\newtheorem{Th}{Theorem}[section]
\newtheorem{Prop}[Th]{Proposition}
\newtheorem{Corol}[Th]{Corollary}
\newtheorem{Lemma}[Th]{Lemma}

\newtheorem{theorem*}{Theorem}[]

\theoremstyle{definition}

\newtheorem{Def}[Th]{Definition}

\newtheorem{Ex}[Th]{Example}

\theoremstyle{remark}

\newcommand{\ds}{\displaystyle}

\newcommand{\R}{{\mathbb R}}

\newcommand{\N}{{\mathbb N}}
\newcommand{\rank}{{\rm rank\,}}
\newcommand{\de}{{\partial\,}}

\newcommand{\id}{{\rm id\,}}
\newcommand{\Vol}{{\rm Vol\,}}

\begin{document}
\date{}
\title{\textsc{\Large{\bf BOUNDS OF HAUSDORFF MEASURES\\OF TAME SETS}}}
%\author{\textbf{Phan Phien}\\ \textit{\small{ Department of Natural Sciences, Nha trang College of Education}}\\ \textit{\small{1 Nguyen Chanh, Nha Trang, Vietnam}}\\
%\textit{\small{e-mail: phieens@yahoo.com}}}
\maketitle
{
{\baselineskip=12pt
\centerline{{Ta L\^e Loi}}

\centerline{\textit{\footnotesize{University of Dalat}}}

\centerline{\textit{\footnotesize{1 Phu Dong Thien Vuong, Dalat, Vietnam}}}

\centerline{\textit{\footnotesize{E-mail: loitl@dlu.edu.vn}}}

\centerline{}

\centerline{{Phan Phien}}

\centerline{\textit{\footnotesize{Nhatrang College of Education}}}

\centerline{\textit{\footnotesize{1 Nguyen Chanh, Nhatrang, Vietnam}}}

\centerline{\textit{\footnotesize{E-mail: phieens@yahoo.com}}}

\centerline{}

\centerline{\textit{\footnotesize{To the memory of Professor Nguyen Huu Duc}}}
}
{\baselineskip=17pt
\begin{abstract}
In this paper we present some bounds of Hausdorff measures of
objects definable in o-minimal structures: sets, fibers of maps,
inverse images of curves of maps, etc. Moreover, we also give some
explicit bounds for semi-algebraic or semi-Pfaffian cases, which
depend only on the combinatoric data representing the objects
involved.

\vspace{8pt}\noindent

\vspace{8pt}\noindent
\end{abstract}
\renewcommand{\thefootnote}{}
\footnote{This research is supported  by Vietnam's National Foundation for Science and Technology Development (NAFOSTED).}
%\footnote{*Corresponding Author.}
\section{Introduction}
Considering the upper bounds for the lengths of curves contained in
a disk, the areas of surfaces in a ball, or generally, the Hausdorff
measures of subsets of a ball, one can see that if the numbers of
points of the intersections of the curves or the surfaces with the
generic lines are bounded, then their lengths or areas could be
estimated (see Figure \ref{h411} for an example). Note that, spirals or oscillations do not have finite
numbers of points of intersections with generic lines, so they can
have infinite lengths in certain disks (see Figure \ref{h412} for an example). The objects of o-minimal
structures have the finiteness of number of connected components
(see \cite{D1}, \cite{D-M}, \cite{C} and \cite{L1}), and
integral-geometric methods allow us to estimate Hausdorff measures
of sets via the numbers of connected components of the intersections
of the sets with generic affine subspaces of appropriate dimensions
(see \cite{F} and Figure \ref{h414} for examples).
%\newpage 
\begin{figure}[htb]
\centering
  % Requires \usepackage{graphicx}
  \includegraphics[width=0.23\textwidth,angle=0]{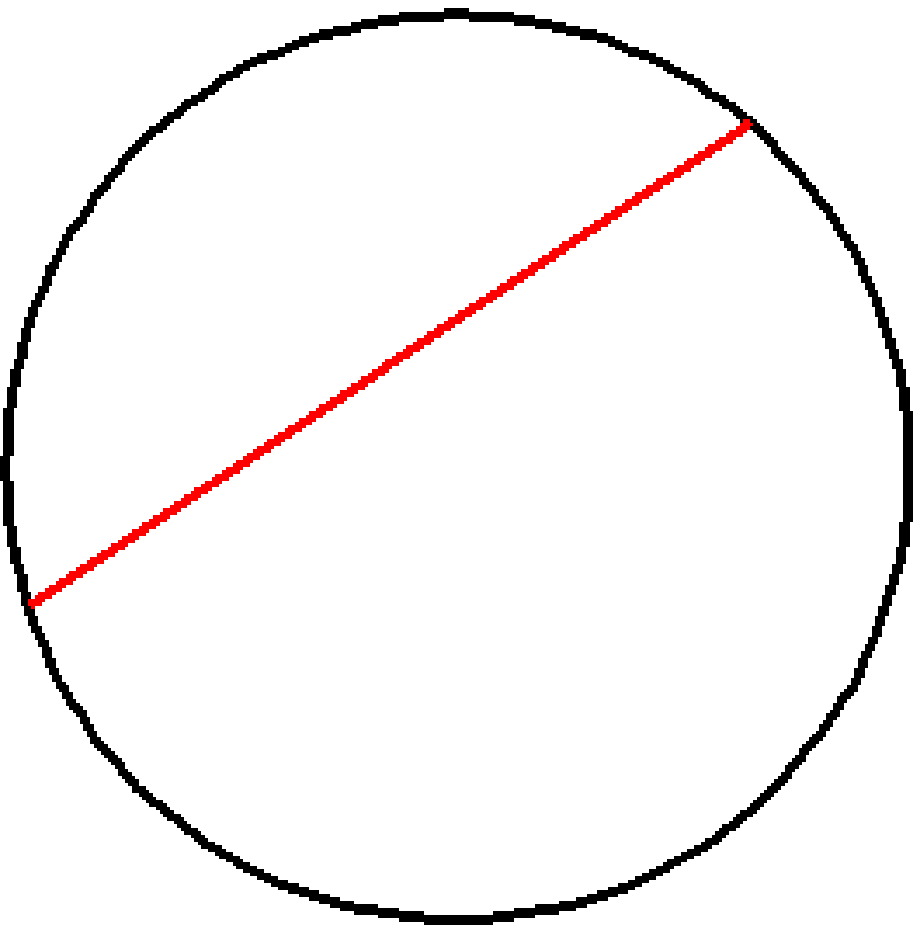}
  %\caption{$l\leq2r$}\label{}
  \includegraphics[width=0.1\textwidth,angle=0]{trang.eps}
  \includegraphics[width=0.23\textwidth,angle=0]{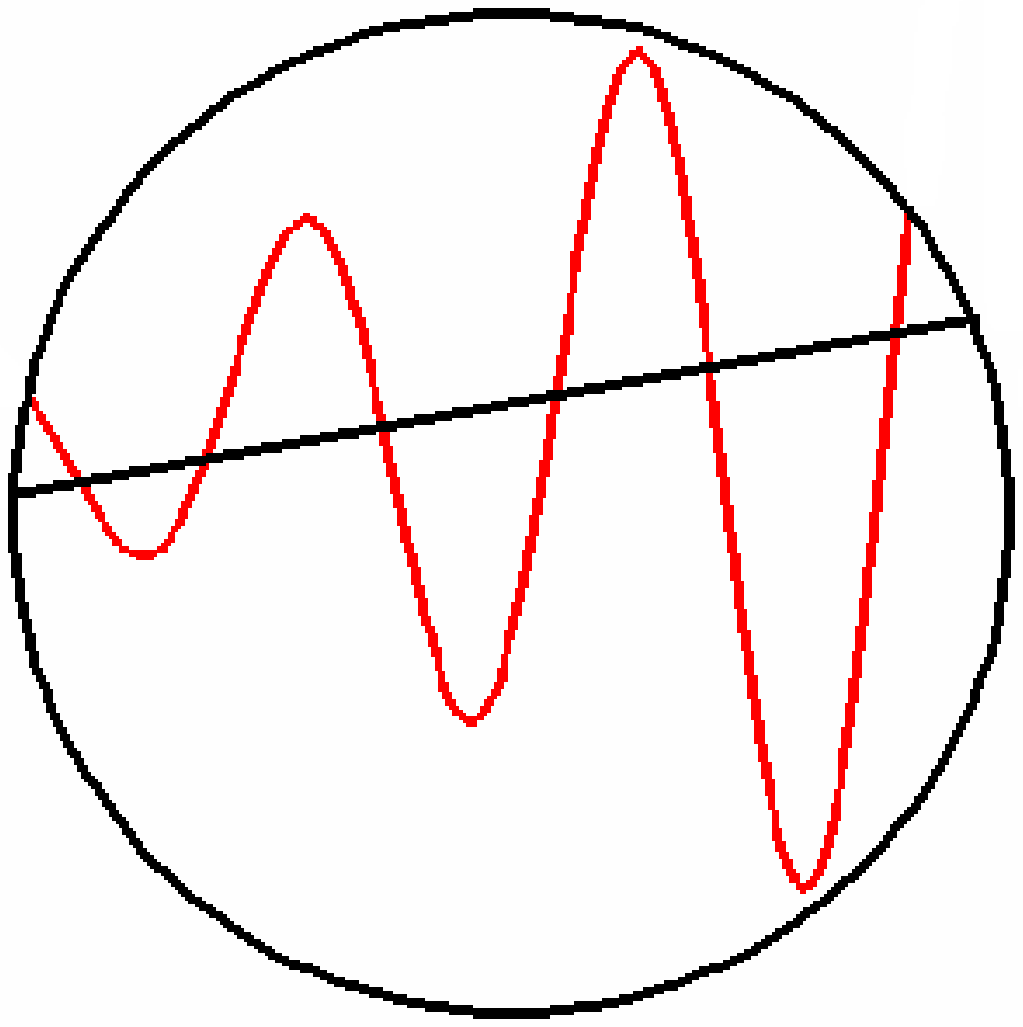}
  \caption{$l\leq2r \ \ \ \ \ \ \ \ \ \qquad $ $l\leq 4dr$}\label{h411}
 \end{figure}
\begin{figure}[htb]
\centering
  % Requires \usepackage{graphicx}
  \includegraphics[width=0.24\textwidth,angle=0]{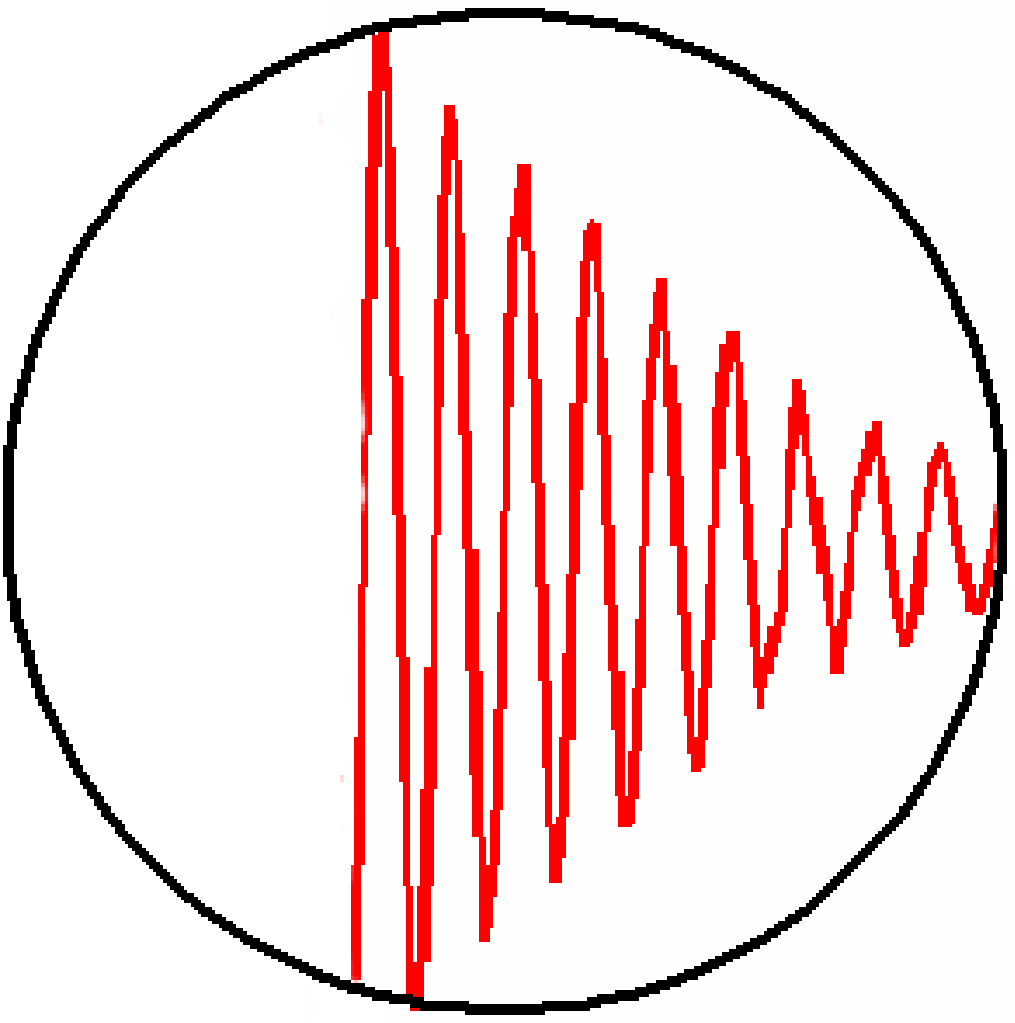}
  %\caption{$l\leq2r$}\label{}
  \includegraphics[width=0.1\textwidth,angle=0]{trang.eps}
  \includegraphics[width=0.24\textwidth,angle=0]{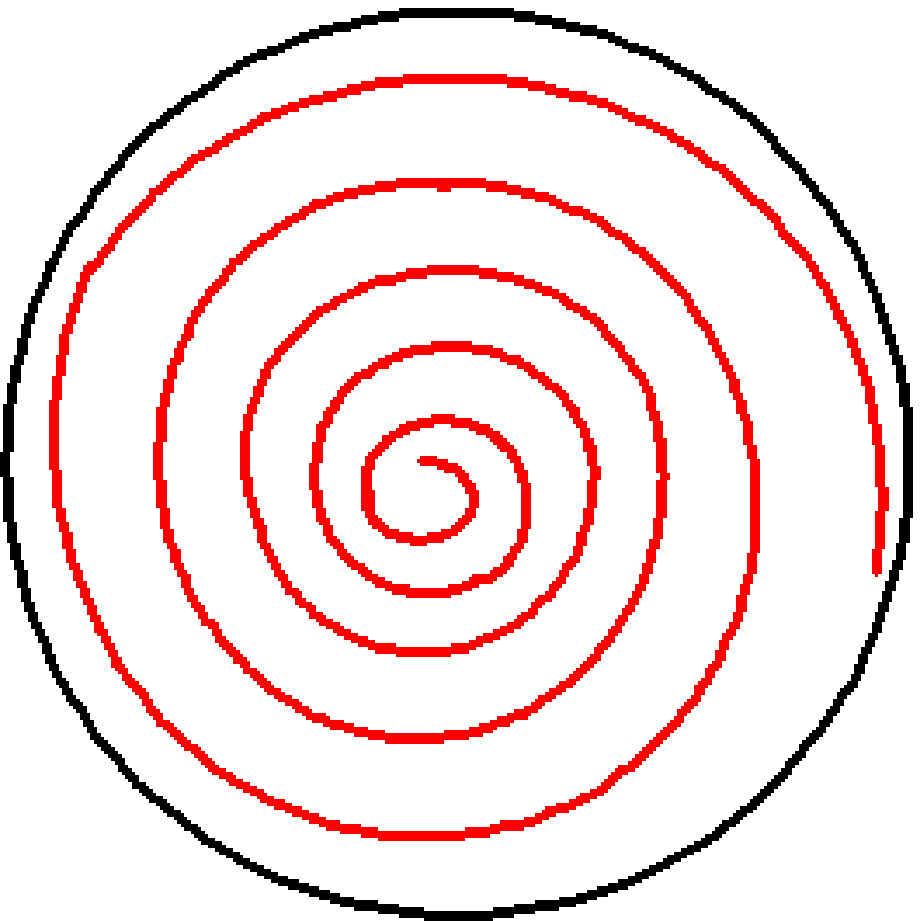}
   \includegraphics[width=0.1\textwidth,angle=0]{trang.eps}
  \includegraphics[width=0.24\textwidth,angle=0]{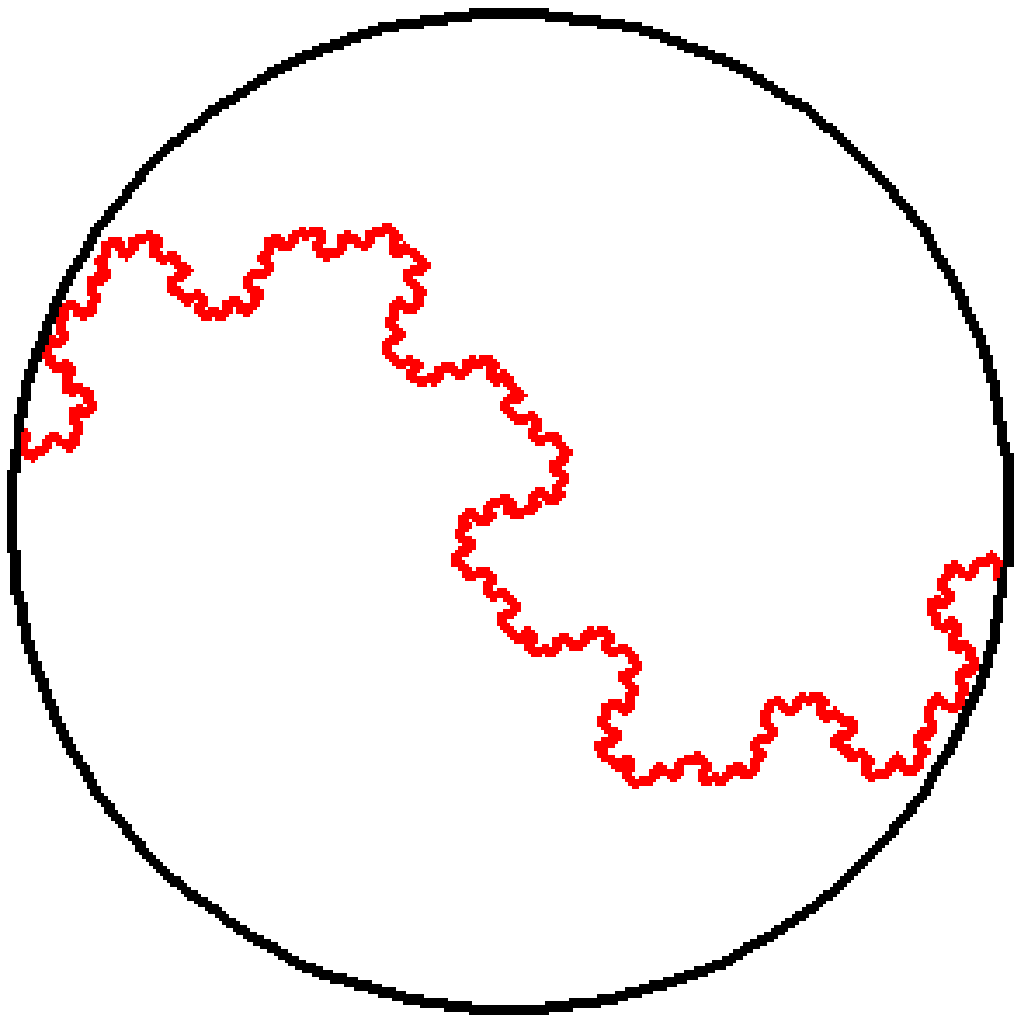}
  \caption{\ $l=\infty$}\label{h412}
\end{figure}
 \begin{figure}[htb]
   \centering
   \includegraphics[width=0.4\textwidth,angle=0]{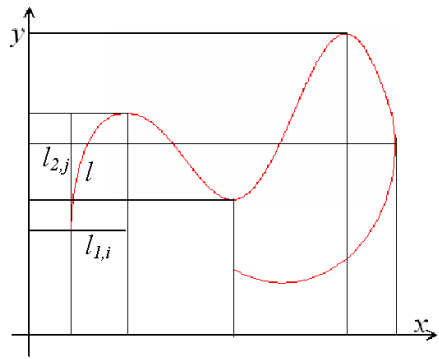}
   \caption{\ $l\leq \sum_i\sum_j (l_{1,i} + l_{2,j})$} \label{h414}
    \end{figure}

  For these reasons, in this paper, we shall use
 integral-geometric methods to give some estimates of Hausdorff
 measures of objects definable in o-minimal structures:
  sets, fibers of maps, inverse images of curves of maps, etc. They can be considered as a
  generalization and refinement of some results of \cite{H}. Moreover, we also give some
  explicit
 bounds for semi-algebraic and semi-Pfaffian cases (relying on the
 results in \cite{B-P-R}, \cite{G-V},\cite{K},\cite{G-V-Z}, and
 \cite{Z}), which depend only on the combinatoric data representing
 the objects involved. These relate to some results in \cite{Y},
 \cite{Y-C}, \cite{D2} and \cite{D-K}.

In section 2 we shall give some definitions. The results and
examples will be stated and proved in sections 3 - 6.
%\noindent
\section{Definitions}
We give here some definitions and notations that will be used later.

\begin{Def}\label{O-minimal structures.} An {\it o-minimal structure}
on the real field $(\R,+,\cdot)$ is a sequence ${\mathcal
D}=({\mathcal D}_n)_{n\in\N}$ such that the following  conditions
are satisfied for all $n\in\N$:
\begin{itemize}
\item  ${\mathcal D}_n$ is a Boolean algebra of subsets of $\R^n$.
\item If $A\in{\mathcal D}_n$, then $A\times \R$ and $\R\times A \in{\mathcal D}_{n+1}$.
\item If $A\in{\mathcal D}_{n+1}$, then $\pi(A)\in{\mathcal D}_n$, where $\pi:\R^{n+1}\to\R^n$ is the projection on the first $n$ coordinates.
\item ${\mathcal D}_n$ contains $\{x\in\R^n: P(x)=0 \}$,  for every polynomial
$P\in \R[X_1,\ldots,X_n]$.
\item Each set in ${\mathcal D}_1$ is a finite union of intervals and points.
\end{itemize}
A set belonging to ${\mathcal D}$ is said to be  {\it definable } (in that structure).
{\it Definable  maps } in structure ${\mathcal D}$ are maps whose graphs are definable
sets in  $\mathcal D$.\\
The class of semi-algebraic sets and the class generated by
semi-Pfaffian sets (\cite{K} and \cite{W}) are examples of such
structures, and there are many interesting classes of sets which
have been proved to be o-minimal. For important properties of
o-minimal structures we refer the readers to \cite{D1}, \cite{D-M},
\cite{C}, \cite{L1} and \cite{W}.
 Note that by Cell Decomposition \cite[Chapter 3 Theorem 2.11]{D1}, the dimension of a definable set $A$
 is defined by
\[ \dim A=\max\{ \dim C: C \textrm{ is a } C^1 \textrm{ submanifold contained in } A\}.\]
In this note we fix an o-minimal structure on $(\R,+,\cdot)$.
``Definable'' means definable in this structure. \end{Def}

\begin{Def}\label{Diagrams of semi-algebraic sets} Let $A \subset \R^m$
be a semi-algebraic set represented by $A =
\cup_{i=1}^p\cap_{j=1}^{s_i}A_{ij}$, where each $A_{ij}$ has the
form:
\[\quad\quad\{x \in \R^m: p_{ij}(x) \star 0\},\]
%~ or\\
%$\textrm{ }\quad\quad \{(x_1,\ldots, x_m) \in \R^m: p_{ij}(x_1,\ldots, x_m) \geq 0\},$\\
where $p_{ij}$ is a polynomial of degree $d_{ij}$ and $\star \in \{>, =\}$.\\
The set of data $D = D(A) = (m, p, s_1, \ldots, s_p, (d_{ij},
i=1,\cdots,p, j=1,\cdots,s_i))$
 is called the \textit{diagram} of the set $A$.
 \end{Def}

 \begin{Def}\label{Formats of Semi-Pfaffian sets}
A \textit{Pfaffian chain} of \textit{length} $l \geq 0$ and
\textit{degree} $\alpha \geq 1$ in an open domain $U \subseteq \R^m$
is a sequence of analytic functions $f = (f_1, \ldots, f_l)$ in $U$
satisfying a system of
Pfaffian equations %be a sequence of real analytic functions defined on a domain $U \subseteq \R^m$. We say that they constitute a  if there exists polynomials $P_{ij}$, each in $n + i$ indeterminates, such that the following equations
\[\frac{\de f_i}{\de x_j}(x) = P_{ij}(x, f_1(x), \ldots, f_i(x)),\
\forall x\in U \ \ (1 \leq i \leq l, ~1 \leq j\leq m). \label{de1}\]
where  $ P_{ij}$ are polynomials  of degree not exceeding $\alpha$.\\
 %Let $f = (f_1, \ldots, f_l)$ be a fixed Pfaffian chain, and $q(x)$ be an analytic function on the domain of that chain.
We say that $q$ is a \textit{\textit{Pfaffian function}} of degree $\beta$ with the Pfaffian
chain $f$ if there exists a polynomial $Q$ of degree not exceeding $\beta$ such that %$q = Q(x, f)$, i.e.
\[q(x) = Q(x, f_1(x), \ldots, f_l(x)), ~~\forall x \in U.\]
%If $Q$ is of degree $\beta$, then we say that $\beta$ is the degree of $q$ in $f$, and we will write $\beta = \textrm{deg}_f(q)$.\\
Let $\mathcal{P} = \{p_1, \ldots, p_s\}$ be a set of Pfaffian
functions. A quantifier-free \textit{formula} (QF formula) with
atoms in $\mathcal{P}$ is constructed as follows:
\begin{itemize}
 \item An atom is of the form $p_i \star 0$,
 where $ 1 \leq i \leq s$ and $\star \in \{>,=,< \}$. It is a QF formula;
 \item If $\Phi$ and $\Psi$ are QF formulae, then their conjunction $\Phi \wedge
 \Psi$, their disjunction $\Phi \vee \Psi$, and the negation $\neg\Phi$ are QF formulae.
\end{itemize}
A set $A \subseteq U$ is called \textit{semi-Pfaffian} if there
exists a finite set $\mathcal{P}$ of Pfaffian functions and a QF
formula $\Phi$ with atoms in $\mathcal{P}$ such that
\[A = \{x \in U :  \Phi(x)\}.\]
Let $A$ be a semi-Pfaffian set as above. Then the \textit{format} of
$A$ is the set of data $F = F(A) = (m, l, \alpha, \beta, s)$,
%be a semi-Pfaffian set defined by a QF formula $\Phi$.
where $m$ is the number of variables, $l$ is the length of $f$,
$\alpha$ is the maximum of the degrees of the polynomials $P_{ij}$,
$\beta$ is the maximum of the degrees of the functions in
$\mathcal{P}$, and $s$ is the number of the functions in
$\mathcal{P}$.
\end{Def}

\begin{Def}\label{A formula of integral geometric measure} Let $m$ be a
positive integer. For each $k\in\{0,\ldots, m\}$, let
$\mathcal{H}^k(A)$ denote the k-dimensional Hausdorff measure of
$A\subset \R^m$. Let $O^*(m,k)$ denote the space of all orthogonal
projections of $\R^m$ onto $\R^k$, i.e.
\[ O^*(m,k)=\{p |\ p:\R^m\to\R^k \textrm{ linear and } p\circ
p^*=\id_{\R^k}\}.\] The orthogonal group $O(m)$ acts transitively on
$O^*(m,k)$ through right multiplication. This action induces a
unique invariant measure $\theta^*_{m,k}$ over $O^*(m,k)$ with
$\theta^*_{m,k}[O^*(m,k)]=1$.\\
\\
\textbf{The Cauchy-Crofton formula}. Since definable sets can be
partitioned into finitely many $C^1$ submanifolds, by \cite[Theorems 2.10.15 and 3.2.26]{F}, for every $k$-dimensional definable bounded
subset $A$ of $\R^m$, we have
\[ \mathcal{H}^k(A)=c(m,k)\int_{O^*(m,k)}\int_{\R^k}\#(A\cap
p^{-1}(y))dy d\theta^*_{m,k}p,\]
 where $\ c(m,k)=\ds\frac{\Gamma(\frac{m+1}2)\Gamma(\frac
 12)}{\Gamma(\frac{k+1}2)\Gamma(\frac{m-k+1}2)}$, and
 $\Gamma(s)=\int_0^{+\infty}e^{-t}t^{s-1}dt \ (s>0)$.
 \end{Def}

\section{Uniform bounds of the Betti numbers of the fibers}
\begin{Prop}\label{Prop 1}
Let $f: A \rightarrow \mathbb{R}^n$ be a continuous definable map. Let $i \in \mathbb{N}$. Then there exists a positive number $M_i$, such that the i-th Betti numbers of the fibers of $f$ are bounded by $M_i$
\[B_i(f^{-1}(y)) \leq M_i, ~\textrm{for all }~ y \in \mathbb{R}^n.\]
In particular, the numbers of connected components of the fibers of $f$ are uniformly bounded.\\
Moreover, if $f$ is semi-algebraic (resp. semi-Pfaffian), then $M_i$
only depends on the diagram (resp. the format) of $f$.
\end{Prop}

\begin{proof} The first part follows from
Hardt's Trivialization Theorem \cite[Chapter 9 Theorem 1.2]{D1}. When $f$ is
semi-algebraic or semi-Pfaffian, the last assertion follows from
\cite{B-P-R},\cite{G-V} or \cite{K},\cite{Z},\cite{G-V-Z}.
\end{proof}

\section{Hausdorff measures of definable sets}
Let $A$ be a subset of $\mathbb{R}^m$. For each $k \in \{0,\ldots,
m\}$, define
\[B_{0,m-k}(A)=\sup \{B_0(A\cap p^{-1}(y)): p\in  O^*(m,k), y \in \mathbb{R}^k\} \]
Note that if $A$ is definable, then applying Proposition \ref{Prop
1} to the canonical projection
\[\{(x,p,y)\in A\times O^*(m,k)\times\R^k: p(x)=y\}\to \{(p,y)\in
O^*(m,k)\times\R^k\},\]
 we  get the boundedness of $B_{0,m-k}(A)$. Moreover, if $A$ is semi-algebraic or semi-Pfaffian, then
$B_{0,m-k}(A)$ is bounded by an explicit constant depending only on
the diagram or the format of $A$ (see the examples below).

\begin{Th}\label{Th31} Let $A, B$ be definable subsets of $\R^m$. Suppose $B$
is compact, $\dim A=k$, and $A\subset B$. Then
\[\mathcal{H}^k(A) \leq c(m,k) B_{0,m-k}(A)\ds\sup_{p\in O^*(m,k)}
\mathcal{H}^k(p(B)).\]
 If moreover $A, B$ are semi-algebraic or
semi-Pfaffian sets, then\
\[\mathcal{H}^k(A) \leq C \ds\sup_{p\in O^*(m,k)}
\mathcal{H}^k(p(B)),\]
 where $C$ is a constant
depending only on the diagram or the format of $A$.
\end{Th}

\begin{proof}
Let $p \in O^*(m, k)$. Set
\[S_p(d) = \{w \in \mathbb{R}^k: \dim(A \cap p^{-1}(w)) = d\}.\]
Applying \cite[Chapter 4 Corollary 1.6]{D1}, we have
\[\dim(A \cap p^{-1}(S_p(d))) = \dim(S_p(d)) + d.\]
Furthermore,
\[\dim(A \cap p^{-1}(S_p(d))) \leq \dim A = k.\]
So if $\dim(S_p(d)) = k$ then
\[d \leq 0.\]
Therefore, for each $p\in O^*(m, k), \ \dim (A\cap p^{-1}(w)) \leq 0$, for all $w
\in \mathbb{R}^k$ outside a definable set of dimension less than
$k$. By the Cauchy-Crofton formula, we get the estimate
\begin{eqnarray}
\mathcal{H}^k(A) & = & c(m,k)\int_{O^*(m,k)}\int_{\R^k}\#(A\cap
p^{-1}(w))dw d\theta^*_{m,k}p \nonumber \\
& \leq & c(m,k)B_{0,m-k}(A)\int_{O^*(m,k)}\ds\int_{\R^k}1_{p(A)}dw d\theta^*_{m,k}p \nonumber \\
& \leq & c(m,k)B_{0,m-k}(A)\int_{O^*(m,k)}\ds\int_{\R^k}1_{p(B)}dw d\theta^*_{m,k}p \nonumber \\
&\leq & c(m,k)B_{0,m-k}(A)\sup_{p\in O^*(m,k)}\mathcal{H}^k(p(B)).
\nonumber
\end{eqnarray}
The last assertion is followed by Proposition 1.
\end{proof}

\begin{Corol}[c.f. \cite{Y-C} and \cite{D-K}]\label{Corol 1} Let A be a definable
subset of $\mathbb{R}^m$ of dimension k. Then for any ball $\mathbf{B}_r^m$
of radius r in $\mathbb{R}^m$,
\[\mathcal{H}^k(A\cap \mathbf{B}_r^m) \leq c(m,k)B_{0,m-k}(A)\Vol_k(\mathbf{B}_1^k)r^k.\]
\end{Corol}

\begin{proof}
From Theorem \ref{Th31}, we get
\[
\mathcal{H}^k(A\cap \mathbf{B}_r^m)  \leq   c(m,k)B_{0,
m-k}(A)\mathcal{H}^k(\mathbf{B}_r^k) = c(m,k)B_{0,m-k}(A)\Vol_k(\mathbf{B}_1^k)r^k.
\]
\end{proof}

\begin{Ex}\label{Example 1}
\item{\bf Algebraic case.} When $A \subset \mathbb{R}^m$ is a
$k$-dimensional algebraic set of degree $d$, then
\[\mathcal{H}^k(A\cap \mathbf{B}_r^m) \leq c(m,k)d\Vol_k(\mathbf{B}_1^k)r^k.\]
In particular, when A is an algebraic curve of degree d in the
plane, then the length $l(A \cap \mathbf{B}_r^2) \leq c(2,1)d2r=\pi dr$.
\item{\bf Semi-algebraic case.} Generally, when $A \subset \R^m$ is a
k-dimensional semi-algebraic set of diagram $D = (m, p,d,s),$ then
\[\mathcal{H}^k(A\cap \mathbf{B}_r^m) \leq c(m,k)B_0(D)\textrm{Vol}_k(\mathbf{B}_1^k)r^k,\]
where \
$B_0(D)=\ds\frac{2^m}{m!}\ds\sum_{i=1}^p((d_is_i)^m+O(s_i^{m-1}))$ ,
with $d_i=\max_{1\leq j\leq s_i}d_{ij}$ and $m$ considered fixed. %(see  \cite{B-P-R} for the more exact formular, c.f. \cite{G-V},\cite{Y-C}).
\item{\bf Semi-Pfaffian case.} We say that $U$ is
a domain of bounded complexity $\gamma$ for the
 Pfaffian chain $f = (f_1, \ldots, f_l)$ if there exists a function g of degree $\gamma$
 in the chain $f$ such that the sets $\{g \geq \varepsilon\}$ form an exhausting family of
 compact subsets of $U$ for $\varepsilon \ll 1$. We call g an exhausting function for $U$.\\
Let $A$ be a k-dimensional semi-Pfaffian set defined by a fixed
Pfaffian chain $f = (f_1, \ldots, f_l)$ of degree $\alpha$ in a
domain $U \subseteq \R^m$ with format $(m, l, \alpha, \beta, s)$,
where $U$ is a domain of bounded complexity $\gamma$ for $f$.
%and $g$ be an exhausting function for $U$ with $deg_fg = \gamma$.
Using \cite[Remark 1.30, Theorem 2.25, and Remark 2.26]{Z} , and applying
Corollary \ref{Corol 1}, we get
\[\mathcal{H}^k(A\cap \mathbf{B}_r^m) \leq c(m,k)(4s+1)^d \mathcal{V}(m, l,
\alpha, \beta^*, \gamma)\Vol_k(\mathbf{B}_1^k)r^k,\] where
\[\mathcal{V}(m, l, \alpha, \beta^*, \gamma) = 2^{\frac{l(l-1)}{2}}
\beta^*(\alpha + \beta^* - 1)^{m-1}\frac{\gamma}{2}[m(\alpha +
\beta^* - 1) + \gamma + \min(m, l)\alpha]^l,\]
 with $\beta^*=\max(\beta, \gamma)$.
\end{Ex}

\section{Uniform bounds of Hausdorff measures of definable fibers}
Let $f: A \rightarrow \R^n$ be a definable map, where
$A\subset\R^m$. For each $k \in \{0, \ldots, \dim A\},$ let
\[I_k(f) = \{y \in \R^n: \dim f^{-1}(y) \leq k\}.\]
Then, by \cite[Chapter 4 Corollary 1.6]{D1}, $I_k(f)$  is definable. Let
\[B_{0,m-k}(f) = \sup\{B_0(f^{-1}(y) \cap p^{-1}(w) \cap \mathbf{B}_r^m(a)): y \in I_k(f),
p\in O^*(m,k),\]
\begin{flushright}
$w \in \R^k, a \in \R^m, r > 0\}.$
\end{flushright}
Note that applying Proposition \ref{Prop 1} to the canonical projection \\
$\{(x,y,p, w, a, r)\in \R^m\times \R^n\times
O^*(m,k)\times\R^k\times\R^m\times\R:$
\begin{flushright}$x\in A, y\in I_k(f),
f(x) = y, p(x) = w, \|x - a\| \leq r\} $ \end{flushright}
\begin{flushright}$\rightarrow \{(y,p,w,a,r)\in \R^n\times O^*(m,k) \times \R^k \times
\R^m \times \R\},$
\end{flushright}
we have the boundedness of $B_{0,m-k}(f)$. When $f$ is
semi-algebraic (resp. semi-Pfaffian), then $B_{0,m-k}(f)$ is bounded
by a constant depending only on the diagram (resp. format) of $f$.

\begin{Th}\label{Th2}
Let $f: A \rightarrow \R^n$ be a continuous definable map, where $A$
is a compact subset of $\R^m$. Then for each $k \in \{0, \ldots,
\dim A\}$, we have
\[\mathcal{H}^k(f^{-1}(y)) \leq c(m,k) B_{0,m-k}(f)\sup_{p\in O^*(m,k)}\mathcal{H}^k(p(A)), ~\textrm{for all}~ y \in I_k(f).\]
In particular, if f is semi-algebraic or semi-Pfaffian map, then
\[\mathcal{H}^k(f^{-1}(y)) \leq C_k \sup_{p\in O^*(m,k)}\mathcal{H}^k(p(A)), ~\textrm{for all}~ y \in I_k(f),\]
where $C_k$ is a constant depending only on the diagram or the
format of $f$.
\end{Th}

\begin{proof}
By \cite[Chapter 4 Proposition 1.5 and Corollary 1.6]{D1}, for each $p\in O^*(m,k)$ and $y \in
I_k(f), \ \dim (f^{-1}(y) \cap p^{-1}(w)) \leq0$, for all $w
\in \R^k$ outside a definable set of dimension less than $k$. By the
Cauchy-Crofton formula, when $y\in I_k(f)$, we get
\begin{eqnarray}
\mathcal{H}^k(f^{-1}(y)) & = &
c(m,k)\int_{O^*(m,k)}\int_{\R^k}\#(f^{-1}(y)\cap p^{-1}(w))dw
d\theta^*_{m,k}p\nonumber \\
& \leq & c(m,k)B_{0,m-k}(f)\int_{O^*(m,k)}\int_{\R^k} 1_{p(A)}dw
d\theta^*_{m,k}p\nonumber \\
&\leq & c(m,k)B_{0,m-k}(f)\sup_{p\in
O^*(m,k)}\mathcal{H}^k(p(A)).\nonumber
\end{eqnarray}
If $f$ is semi-algebraic or semi-Pfaffian, then using the note above
we have the last assertion.
\end{proof}

\begin{Corol}\label{Corol 2}
Let $f: A \rightarrow \R^n$ be a continuous definable map, where
$A\subset\R^m$. Then for each $k \in \{0, \ldots, \dim A\}$ and for
any ball $\mathbf{B}_r^m$ of radius r in $\R^m$,
\[\mathcal{H}^k(f^{-1}(y) \cap \mathbf{B}_r^m) \leq c(m,k) B_{0,m-k}(f)\Vol_k(\mathbf{B}_1^k)r^k, ~\textrm{for all}~ y \in I_k(f).\]
In particular, if f is semi-algebraic or semi-Pfaffian map, then
\[\mathcal{H}^k(f^{-1}(y) \cap \mathbf{B}_r^m) \leq C_k r^k, ~\textrm{for all}~ y \in I_k(f),\]
where $C_k$ is a constant depending only on the diagram or the
format of $f$.
\end{Corol}

\begin{Ex}\label{Example 2.}
%\item{\textbf{Fewnomial case.}}
Let $\alpha_1,\ldots, \alpha_q\in\N^m$.
Consider the family of algebraic surfaces  in the positive orthant
determined by  the `fewnomials' having only at most the monomials $
x^{\alpha_i}, i=1,\ldots, q$:
 \[A=\{(x,a): x=(x_1,\ldots,x_m)\in\R^m, a=(a_1, \ldots, a_q)\in\R^q, \]
\[x_1>0,\ldots, x_m>0, \ds\sum_{i=1}^qa_ix^{\alpha_i}=0\}.\]
Let $f$ be the projection $(x,a)\mapsto a$ and $A_a=A\cap f^{-1}(a)$.\\
When $k=m-1$, and $\dim A_a\leq m-1$ from the theorem we have the
following estimates:
\\
\underline{Estimate 1}. Since $A_a$ is a
semi-algebraic set of diagram $(m, 1, m+1, (1,\cdots,1,d))$, with $d
= \max_i |\alpha_i|$, using the Oleinik-Petrovskii-Thom-Milnor bound
(see \cite{O-P},\cite{Th}, \cite{M}), we get
\[ {\mathcal H}^{m-1}(A_a\cap \mathbf{B}_r^m)\leq c(m,m-1)B_0(D(A_a))\textrm{Vol}_{m-1}(\mathbf{B}_1^{m-1}) r^{m-1},\]
where $B_0(D(A_a)) \leq \frac{1}{2}(m+d)(m+d-1)^{m-1}$. %(see also \cite{B-P-R} for the better bound).
\\
\underline{Estimate 2}.  Using the Khovanskii bound \cite[Chapter III Corollary 5]{K}, we get
\[ {\mathcal H}^{m-1}(A_a\cap \mathbf{B}_r^m)\leq c(m ,m-1)B_{0}(f)\textrm{Vol}_{m-1}(\mathbf{B}_1^{m-1}) r^{m-1},
\]
where $B_0(f)\leq 2^{\frac{q(q-1)}2}(2m)^{m-1}(2m^2-m+1)^q$. \\
\end{Ex}

A family $(C_q)_{q\in Q}$ is called a {\it definable family of
definable curves} in $B\subset \R^n$ if there exists a definable map
$\gamma: Q\times [0,1]\to B$, $\gamma(q,t)=\gamma_q(t)$, such that
for each $q\in Q$, $\gamma_q:[0,1]\to B$ is continuous, injective
and $C_q=\gamma_q([0,1])$.

Let $\Phi^1$ denote the set of all odd, strictly increasing $C^1$
definable bijection from $\R$ to $\R$ and flat at 0.

\begin{Th}\label{Th3}
Let $f: A \rightarrow \R^n$ be a continuous definable map, and
$A\subset \R^m$ be a compact set.  Then for each $k \in \{0, \ldots,
\dim A\}$, compact definable subset $B$ of $I_k(f)$ and definable
family of definable curves $(C_q)_{q\in Q}$ in $\R^n$, there exists
$\varphi \in \Phi^1$, such that
\[\mathcal{H}^{k+1}(f^{-1}(C_q\cap B) ) \leq \varphi^{-1}(\mathcal{H}^1(C_q)),
~\textrm{for all }~ q\in Q.\]
 In particular, if f is semi-algebraic,
and $(C_q)_{q\in Q}$ is a semi-algebraic family  of semi-algebraic
curves, then there exist $C, \alpha>0$ such that
\[\mathcal{H}^{k+1}(f^{-1}(C_q\cap B)) \leq C(\mathcal{H}^1(C_q))^\alpha,
 ~\textrm{for all }~ q\in Q.\]
%where $\alpha$ depends only on the diagram of $f$.
\end{Th}

\noindent First we have:

\begin{Lemma}\label{Le45} Let $h:B\to \R^m$ be a continuous
definable map, and $B$ be a compact subset of $\R^n$. Then there
exists $\psi\in\Phi^1$, such that
\[\mathcal{H}^{1}(h(C_q\cap B) ) \leq \psi^{-1}(\mathcal{H}^1(C_q)),
~\textrm{for all }~ q\in Q.\]
\end{Lemma}
\begin{proof}[Proof of Lemma \ref{Le45}]
To prove the lemma for a family $(C_q)_{q\in Q}$ of
curves in $B$, applying \cite[C.17]{D-M} and the uniform bounds \cite[4.4]{D-M} to the family
\[(\{ t\in [0,1]: \exists i,  h_i\circ\gamma_q \textrm{ is not strictly monotone on
any neighbourhood of } t\ \})_{q\in Q},\]
where $h=(h_1,\cdots,h_m)$,  we have $\psi_1\in\Phi^1$
such that
\[\mathcal{H}^{1}(h(C_q) ) \leq \psi_1^{-1}(\mathcal{H}^1(C_q)),
~\textrm{for all}~ q\in Q.\]
 For a family of definable curves in $\R^n$, the number of connected components of $C_q\cap B$ is
 uniform bounded by $M$, for all $q\in Q$. Therefore, denoting the
connected components of $C_q\cap B$ by $C_{q,i}$ and applying the
above case, we get
\[\begin{array}{llll}
\mathcal{H}^{1}(h(C_q\cap B) ) &\leq\
\sum_i\mathcal{H}^1(h(C_{q,i}))
\leq\ \sum_i\mathcal\psi_1^{-1}(\mathcal{H}^1(C_{q,i}))\\
& \leq\ M\psi_1^{-1}(\mathcal{H}^1(C_{q})) \leq\
\psi^{-1}(\mathcal{H}^1(C_q)), ~\textrm{for all}~  q\in Q,
\end{array}
\]
 where $\psi\in \Phi^1$, $\psi(t)=\psi_1(t/M)$.
\end{proof}
\begin{proof}[Proof of Theorem \ref{Th3}]
The proof of the theorem is an adaptation of that of \cite[Theorem 5]{H}.\\
For $k = 0$ : Since $B$ is compact and the fibers of $f$ over $B$
are finite, by Trivialization \cite[Chapter 9 Theorem 1.2]{D1}, $f^{-1}(B) =
\cup_{j=1}^JA_j$, where $A_j$ is a definable compact set, and
$f|_{A_j}$ is injective. For each $j \in \{1, \ldots, J\}$, applying
the lemma to $(f|_{A_j})^{-1}$, we get $\psi_j \in \Phi^1$, such
that
\[
\mathcal{H}^1((f|_{A_j})^{-1}(C_q\cap B))\ \leq
\psi_j^{-1}(\mathcal{H}^1(C_q)), \ \ ~\textrm{for all } q\in Q.\]
 So
 \[
\mathcal{H}^1(f^{-1}(C_q\cap B))\ \leq\
\sum_j\varphi_j^{-1}(\mathcal{H}^1(C_q))\ \leq \
\varphi^{-1}(\mathcal{H}^1(C_q)), \ \ ~\textrm{for all } q\in Q,\]
where $\varphi\in \Phi^1$ with $\varphi^{-1}\geq
\ds\sum_{j=1}^J\varphi_j^{-1}$.
\\
 For $k \geq 1$: let $G_k(\R^m)$ denote the Grassmannian
of $k$-dimensional linear subspaces of $\R^m$. Define
\[\textrm{dist}(L, L') = \sup\{d(x, L'): x \in L, \|x\| = 1\}, ~\textrm{for}~ L \in G_k(\R^m), L' \in G_l(\R^m).\]
Let $\pi:\R^m\to\R^k$ denote the canonical projection. Choose a
finite subset $I$ of $O(m)$ and $\delta > 0$, so that for each $L
\in G_k(\R^m)$, there exists $g \in I$ so that
\[\textrm{dist}(L, (\pi\circ g)^{-1}(0)) > \delta.\]
By \cite{L2} we can choose a stratification $\mathcal{S}$ of $A$
satisfying Whitney's condition (a), so that for each $S\in
\mathcal{S}$, rank$f|_S$ is constant and either $f(S)\subset I_k(f)$
or $f(S)\cap I_k(f)=\emptyset$. Let $\mathcal{J} = \{S \in
\mathcal{S}: \dim S - \rank f|_S = k\}$. We can refine the
stratification so that for each $g \in I$ and $T\in \mathcal{J}$,
the definable function
\[d(T, g)(x) = \textrm{dist}(T_x(T\cap f^{-1}(f(x))), (\pi \circ g)^{-1}(0)) - \delta\]
has a constant sign on $T$.\\
For each $ S \in \mathcal{S}\backslash \mathcal{J}$ we have $\dim (S
\cap f^{-1}(y)) \leq k - 1$  for all  $y \in I_k(f)$, therefore,
$\mathcal{H}^{k+1}(f^{-1}(C_q\cap I_k(f))\backslash \cup_{T\in
\mathcal{J}}T) = 0$ whenever $q\in Q$. \\
For each $T\in \mathcal{J}$, there is a $g_T \in I$ so that $d(T,
g_T)$ is positive on $T$. Hence, by Whitney's condition (a),
$\dim(f^{-1}(y) \cap (\pi \circ g_T)^{-1}(w) \cap \textrm{cl}(T))
\leq 0, ~\textrm{for all}~ y \in I_k(f), w \in \R^k.$\\
 For each $g \in I$, let $A_g =
\cup\{\textrm{cl}(T): T \in \mathcal{J}, g_T =g\}.$ Using the coarea
formula \cite[Theorem 3.2.22 (3)]{F} and applying case $k = 0$ with $A :=
A_g$, $f:= (f, \pi\circ g)|_{A_g}$, and $((C_q\times w))_{(q, w)\in
Q\times \R^k}$ for the family of curves, we get
\begin{eqnarray}
\mathcal{H}^{k+1}(f^{-1}(C_q\cap B))
& \leq & \sum_{g \in I}\mathcal{H}^{k+1}(g(A_g\cap f^{-1}(C_q\cap B))) \nonumber \\
& = &\sum_{g \in I}\int_{g(A_g\cap f^{-1}(C_q\cap B))}\
d\mathcal{H}^{k+1} \nonumber \\
& = &\sum_{g \in I}\int \mathcal{H}^1(g(A_g\cap f^{-1}(C_q\cap B))
\cap \pi^{-1}(w)) dw \nonumber\\
& = &\sum_{g \in I}\int \mathcal{H}^1(A_g\cap f^{-1}(C_q\cap B))\cap g^{-1}(\pi^{-1}(w))) dw \nonumber \\
& = & \sum_{g \in I}\int \mathcal{H}^1(A_g \cap (f, \pi \circ g)^{-1}[(C_q\times w)\cap(B\times \pi(g(A)))])dw \nonumber \\
\nonumber
%\end{eqnarray}
%\\
%\begin{eqnarray}
%\phantom{\mathcal{H}^{k+1}(f^{-1}(C_q\cap B))}
&\leq & \sum_{g \in I} \int 1_{\pi\circ g(A)}\varphi^{-1}(\mathcal{H}^1(C_q))dw\\
&\leq & \sum_{g \in I} \mathcal{H}^k(\pi\circ g(A))
\varphi^{-1}(\mathcal{H}^1(C_q)) \nonumber
\\
&\leq & \bar{\varphi}^{-1}(\mathcal{H}^1(C_q)) \ , \textrm{ for all
} q\in Q, \nonumber
\end{eqnarray}
where  $\bar{\varphi}\in\Phi^1$ of the form $\bar{\varphi}(t)=\varphi(t/K)$. \\
 If $f$ is a semi-algebraic map, then by
the \L ojasiewicz inequality $\varphi$ has the form
$\varphi^{-1}(y)=C\|y\|^{\alpha}$.
%Moreover, by \cite{B-R} Remark
%2.3.13, $\alpha$ can be effectively bounded by the diagram of $f$.
%The last assertion follows.
\end{proof}

Note that the above estimate is `effective', not explicit.

\begin{Ex}\label{Example 3.}
\item a) Applying the theorem to the family of segments, we get
$\varphi\in \Phi^1$, such that
\[\mathcal{H}^{k+1}(f^{-1}([y,z] \cap B)) \leq \varphi^{-1}(\|y-z\|), \textrm{whenever } y,z\in\R^n.\]
In particular, if $f$ is a semi-algebraic or semi-Pfaffian map, then
there exist $C,\alpha>0$ such that
\[\mathcal{H}^{k+1}(f^{-1}([y,z]\cap B)) \leq C\|y-z\|^{\alpha}.\]
\item b) In general, for the semi-algebraic case one can not choose
$C$ depending only on the diagram of $f$, or $\alpha=1$ in the
estimate of the preceding theorem, e.g. for $f_k(x)=kx^n$\ with
$n\geq 2,
k>0$, %there does not exist any $C>0$ such that the length
${\mathcal H}^1(f_k^{-1}([0,y])=\ds\frac 1{\sqrt[n]{k}}\sqrt[n]{y}$,
for every $y>0$.
\\
\item c) Let $f(x)=e^{-\frac 1{|x|}}$. Then $f$ is definable in the
o-minimal structure $\R_{exp}$, and $f^{-1}([0,y])=[0,-\ds\frac
1{\ln|y|}]$. Since $\ds\frac 1{y^{\alpha}\ln|y|}\to \infty$, when
$y\to 0$, there does not exist $C, \alpha>0$ so that ${\mathcal
H}^1(f^{-1}([0,y])\leq C|y|^{\alpha}$ for all $y\in [0,1]$.
\end{Ex}

\section{Morse-Sard's Theorem}

\begin{Th}
Let $f: A \rightarrow \R^n$ be a definable map. Suppose $A = \cup_{i
\in I}C_i$ is a finite union of $C^1$ definable manifolds $C_i$,
such that $f|_{C_i}$ is of class $C^1$. For each $s \in \N$ and $i
\in I$, let
\[\Sigma_s(f, C_i) = \{x \in C_i : \rank df|_{C_i}(x) < s\}~ \textrm{and}~ \Sigma_s(f, A) = \bigcup_{i \in I}\Sigma_s(f, C_i)\]
Then $C_s(f, A) = f(\Sigma_s(f, A))$ is a definable set of dimension
$< s$. In particular, $\mathcal{H}^s(C_s(f, A)) = 0$.
\end{Th}

\begin{proof}
The proof is similar to \cite{L3}.
\end{proof}

\section{Remarks} The results in this paper still hold true
for tame sets (see  \cite{D-M}, \cite{S}, \cite{T} for the
definitions) with global changing to local. Applying theorems 1 and
2, one can get the explicit estimates for sub-Pfaffian case (see
\cite{G-V-Z}).
}
\bibliographystyle{amsalpha}

}
\end{document}